\documentclass[a4paper,10pt]{amsart}

\usepackage{hyperref,pgfplots}

\usepackage[T1]{fontenc}
\usepackage[sc]{mathpazo}

\usepackage{graphicx}
\usepackage{amssymb}
\usepackage{mathrsfs}
\usepackage{color}
\usepackage[all]{xy}

\newcommand{\R}{\mathbb R}

\newcommand{\N}{\mathbb N}

\newcommand{\TC}[2]{{T}_{#1}^{#2}}

\newcommand{\blob}{\bullet}

\newcommand{\figdir}{.}

 \DeclareMathOperator{\Hom}{Hom}

\DeclareMathOperator{\Card}{\#}
\DeclareMathOperator{\GL}{GL}
\DeclareMathOperator{\adj}{adj}
\DeclareMathOperator{\intrinsic}{V}

\renewcommand{\paragraph}[1]{\bigskip\noindent \emph{#1.}}

 \newtheoremstyle{thm2}{0.5 em}{0.5 em}%
     {\slshape}
     {}
     {\scshape}
     {.}
     {0.5 em}
     {\thmname{#1}\thmnumber{ #2}\thmnote{ (#3)}}

 \newtheoremstyle{def2}{0.5 em}{0.5 em}%
     {}
     {}
     {\scshape}
     {.}
     {0.5 em}
     {\thmname{#1}\thmnumber{ #2}\thmnote{ (#3)}}

\theoremstyle{thm2}

\newtheorem{thm}{Theorem}
\newtheorem{prop}[thm]{Proposition}

\newtheorem{lemma}[thm]{Lemma}

\newtheorem*{StrongConjecture}{Strong Asymptotic Conjecture}
\newtheorem*{WeakConjecture}{Weak Asymptotic Conjecture}

\theoremstyle{def2}

\renewcommand{\phi}{\varphi}
\renewcommand{\epsilon}{\varepsilon}

\begin{document}
\title[Asymptotic magnitude of subsets of Euclidean space]
{On the asymptotic magnitude of subsets of\\Euclidean space}
\author{Tom Leinster}
\email{Tom.Leinster@glasgow.ac.uk}
\author{Simon Willerton}
\email{S.Willerton@sheffield.ac.uk}
%


\begin{abstract}
Magnitude is a canonical invariant of finite metric spaces which has its origins in category theory; it is analogous to cardinality of finite sets.  Here, by approximating certain compact subsets of Euclidean space with finite subsets, the magnitudes of line segments, circles and Cantor sets are defined and calculated.  It is observed that asymptotically these satisfy the inclusion-exclusion principle, relating them to intrinsic volumes of polyconvex sets.
\end{abstract}

\maketitle

\section*{Introduction}
In \cite{Leinster:Cardinality} one of us introduced the notion of the Euler characteristic
of a finite category and showed how it linked together various notions of size in mathematics, including the cardinality of sets and the Euler characteristics of topological spaces, posets and graphs.   In \cite{Leinster:MetricSpacesBlogPost,Leinster:Magnitude}
it was shown how to transfer this to a notion  of `magnitude'%
\footnote{The terms `Euler characteristic' and `cardinality' could have been used here, as in~\cite{Leinster:Cardinality} and~\cite{Leinster:MetricSpacesBlogPost}, but we
have decided to use a word with less mathematical ambiguity.}
 of a finite metric space, using the fact that a metric space can be viewed as an
enriched category.

One way of viewing magnitude is as the `effective number of points'.  Consider, for example, the $n$-point metric space in which any two points are a distance $d$ apart.  When $d$ is very small, the magnitude is just greater than $1$ --- there is `effectively only one point'.  As $d$ increases, the magnitude increases, and when $d$ is very large, the magnitude is just less than $n$ --- there are `effectively $n$ points'.
The magnitude of a finite metric space actually first appeared in the biodiversity literature~\cite{SolowPolasky:MeasuringBiologicalDiversity}, under the name `effective number of species', although its mathematical properties were hardly explored.
It should be noted that, contrary to the simple example given above, magnitude can display wild behaviour of various types~\cite{Leinster:Magnitude}: when a space is scaled up, its magnitude can sometimes \emph{decrease}, and there are some exceptional finite metric spaces for which the magnitude is not well-defined.

In this paper --- which requires no category theory --- we consider the notion of magnitude for certain non-finite metric spaces, in particular for certain compact subsets of Euclidean space.  This is done by approximating such a subset $A$ with a sequence of finite subsets
of $A$ and taking the limit of
the corresponding sequence of magnitudes.  In the cases we consider here --- circles,
line-segments and Cantor sets --- as the
subset is scaled up this answer behaves like a linear combination of `intrinsic volumes',
such as the length and Euler characteristic, which satisfy the inclusion-exclusion
principle.   This leads us to conjecture that for a subspace $A$ the magnitude $|A|$ decomposes as follows:
   \[|A|=P(A)+q(A)\]
where $P$ is a function, defined on some class of subsets of Euclidean space, which satisfies $P(A\cup
B)=P(A)+P(B)-P(A\cap B)$ and $q(A)$ tends to zero as $A$ is scaled bigger and bigger.
In other words, \emph{the magnitude of subsets of Euclidean space asymptotically satisfies
the inclusion-exclusion principle}.   Whilst the proof of this for some examples requires only elementary analysis, the proof for circles requires more subtle asymptotic analysis.
Empirical calculations for some subsets in two and three dimensions are
consistent with this and appear elsewhere \cite{Willerton:Heuristics}.

In an earlier version of this paper we were less definite in some of our assertions.  We calculated
the limiting value of approximations of each of the spaces and claimed that that \emph{should} be the value of the magnitude of the space in question, and that it should be independent of the choice of approximating sequence.  Motivated by our work, Meckes~\cite{Meckes:PositiveDefinite} proved that this was indeed the case and that our methods do indeed give unique results (see Section~\ref{Section:MeckesResults} below).  Another development since the time of writing is the use of ``weight measures'', this development means that several calculations presented here can be done more quickly, see \cite{Meckes:PositiveDefinite} and \cite{Willerton:Homogeneous}.  However, the methods presented here seem to be more general as it is not known, and seems unlikely, that every compact metric space, or even every compact subset of Euclidean space, admits a weight measure.  

There will now follow a more detailed description of the magnitude of metric spaces,
the inclusion-exclusion principle and intrinsic volumes.

\subsection*{Asymptotic conjectures}
Given a metric space $X$ with $n$ points one can try to associate to it an invariant called the magnitude; the definition is given in Section~\ref{Section:DefinitionOfMagnitude} and the definition is motivated by category theory in Section~\ref{Section:CategoryTheoreticMotivation}.  If $X$ is a metric subspace of some Euclidean space then it will have a well-defined magnitude (see Section~\ref{Section:MeckesResults}).  However, not every finite metric space has a well-defined magnitude, but every `sufficiently separated' one does (Theorem~\ref{Theorem:SufficientlySeparated}). In particular, if for $t>0$ we define $tX$ to be $X$ scaled by a factor of $t$, so that it is a metric space with the same points as $X$ but with the metric defined by $d_{tX}(x,x'):=td_X(x,x')$, then for $t$ sufficiently large the magnitude $|tX|$ is well-defined and $|tX|\to n$ as $t\to \infty$ (Theorem~\ref{Theorem:AsymptoticFinitePoints}).  So asymptotically, the magnitude is the number of points in the metric space.

Whilst the magnitude in the case of \emph{finite} metric spaces is interesting, not least
for its connections with biodiversity measures (see \cite{Leinster:MaximumEntropy}),
in this paper we consider extending this
notion of magnitude to non-finite metric spaces, primarily in the form of compact
subsets of Euclidean
space with the subspace metric.  In the cases we consider here one interesting feature
which emerges is that the magnitude seems to `asymptotically' obey the inclusion-exclusion principle, where `asymptotically' means  with regard to the space being
scaled up larger and larger.

The inclusion-exclusion principle is embodied in the notion of a valuation.
Let ${\mathcal{D}}$ be some class of subsets of Euclidean space closed under binary intersections.  A \emph{valuation} on ${\mathcal{D}}$ is a real-valued function $P$  on ${\mathcal{D}}$ satisfying the following:
\begin{itemize}
\item for all $A,B\in {\mathcal{D}}$ such that $A\cup
B\in \mathcal{D}$ we have \[P(A\cup
B)=P(A)+P(B)-P(A\cap B);\]
\item $P(\emptyset)=0$.
\end{itemize}

%
Following Klain and Rota~\cite{KlainRota:Book}, we will often take $\mathcal{D}$ to be the collection of \emph{polyconvex sets} in $\R^m$ --- that is those subsets of Euclidean
space which are finite unions of compact, convex sets.   Similarly following Klain and Rota, we define an \emph{invariant valuation} on $\R^m$ to be a valuation on the set of polyconvex subsets of $\R^m$ for which the following
axioms hold.
\begin{itemize}
\item It is invariant under rigid motions.
\item It is continuous on \emph{convex} sets with respect to the Hausdorff topology.
\end{itemize}
Examples of such things include the Euler characteristic and the $m$-dimensional
volume.
Hadwiger's Theorem says that the vector space of invariant valuations on $\R^m$ is $
(m+1)$-dimensional and that there is a canonical basis $\{\intrinsic_0,\dots, \intrinsic_m\}$ where
each $\intrinsic_i$ is scaling-homogeneous in the sense that $\intrinsic_i(tA)=t^i\intrinsic_i(A)$ for any
scaling $t>0$ and any polyconvex set $A$, and is normalized on cubes so that $
\intrinsic_i([0,1]^i)=1$.  The valuation $\intrinsic_i$ is called the \emph{$i$th intrinsic volume}.
For a polyconvex set $A$ in $\R^m$ the $m$th intrinsic volume $\intrinsic_m(A)$ is the
usual $m$-dimensional volume of $A$; the  $(m-1)$th intrinsic volume $\intrinsic_{m-1}(A)$
is half of the ``surface area'', that is half of the $(m-1)$-volume of the boundary of $A$;
and the zeroth intrinsic volume $\intrinsic_0(A)$ is the Euler characteristic.  The other
intrinsic volumes are not so well known --- although the first intrinsic volume of a
cuboid,%
\footnote{A three-dimensional rectangular box.}
length plus width plus height, is beloved of those making restrictions for air-travel carry-on luggage.

It might seem that the $i$th intrinsic volume in $\R^m$ should be denoted by say $\intrinsic_i^m$ to remove ambiguity: however, they are normalized so that this is
unnecessary, namely if $A\subset \R^m$ is a polyconvex set which actually is
contained in a subspace $\R^p\subset \R^m$ then $\intrinsic^p_i(A)=\intrinsic^m_i(A)$.

\begin{table}[tb]
\begin{center}
\begin{tabular}{|c|c|l|}
\hline
\multicolumn{2}{|c|}{subspace $A$}&
\multicolumn{1}{|c|}{magnitude $|A|$}\\\hline
$n$ points&\raisebox{-.45\height}{\begin{picture}(0,0)%
\includegraphics{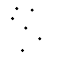}%
\end{picture}%
\setlength{\unitlength}{3947sp}%
\begingroup\makeatletter\ifx\SetFigFont\undefined%
\gdef\SetFigFont#1#2{%
  \fontsize{#1}{#2pt}%
  \selectfont}%
\fi\endgroup%
\begin{picture}(467,461)(142,234)
\end{picture}%
}&$\quad\phantom{\ell/2+{}}n+q_1(A)$\\
ternary Cantor set of length $\ell$&\raisebox{-.45\height}{\begin{picture}(0,0)%
\includegraphics{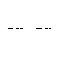}%
\end{picture}%
\setlength{\unitlength}{3947sp}%
\begingroup\makeatletter\ifx\SetFigFont\undefined%
\gdef\SetFigFont#1#2{%
  \fontsize{#1}{#2pt}%
  \selectfont}%
\fi\endgroup%
\begin{picture}(466,460)(143,236)
\end{picture}%
}&$f(\ell)\ell^{\log_32}+q_2(\ell)$\\
closed interval of length $\ell$&\raisebox{-.45\height}{\begin{picture}(0,0)%
\includegraphics{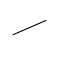}%
\end{picture}%
\setlength{\unitlength}{3947sp}%
\begingroup\makeatletter\ifx\SetFigFont\undefined%
\gdef\SetFigFont#1#2{%
  \fontsize{#1}{#2pt}%
  \selectfont}%
\fi\endgroup%
\begin{picture}(468,462)(143,234)
\end{picture}%
}&$\quad\ell/2 +1$\\
circle of circumference $\ell$&\raisebox{-.45\height}{\begin{picture}(0,0)%
\includegraphics{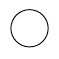}%
\end{picture}%
\setlength{\unitlength}{3947sp}%
\begingroup\makeatletter\ifx\SetFigFont\undefined%
\gdef\SetFigFont#1#2{%
  \fontsize{#1}{#2pt}%
  \selectfont}%
\fi\endgroup%
\begin{picture}(466,460)(143,236)
\end{picture}%
}&$\quad\phantom{{}+1}\ell/2+q_3(\ell)$\\
\hline
\end{tabular}
\end{center}
\caption{The asymptotic behaviour of the magnitude of some subsets of Euclidean
space.  Here $q_1(t A),q_2(t),q_3(t)\to 0$ as $t\to \infty$, and $f$ is a `nearly constant'
function: $f(\ell)\simeq 1.205$ for all $\ell$.  Note that $\log_32$ is the Hausdorff dimension
of the Cantor set.}
\label{table:IntrinsicVolPolyExamples}
\end{table}
In this paper we calculate the magnitude for some subspaces
of
Euclidean space.  The asymptotic behaviour of these is summarized in
Table~\ref{table:IntrinsicVolPolyExamples}.  Looking at the data one might think that,
restricting to polyconvex sets, the magnitude splits as
  \[|A|=P(A)+q(A)\]
where $q(tA)\to 0$ as $t\to\infty$ and $P(A)=\intrinsic_0(A)+\frac12 \intrinsic_1(A)+\text{higher order
terms}$. In fact, plausibility arguments and some rough empirical results given in~\cite{Willerton:Heuristics}, together with some partial results in~\cite[Section~3.5]{Leinster:Magnitude}
 suggest
  \[P(A)=\sum_{i=0}^\infty\frac{\intrinsic_i(A)}{i!\,\omega_i}
  = \intrinsic_0(A)+\frac{\intrinsic_1(A)}{2}+\frac{\intrinsic_2(A)}{2\pi} +\dots,\]
where $\omega_i$ is the volume of the unit $i$-ball in $\R^i$.
It is perhaps worth noting here that the natural normalization of the
$i$-dimensional Hausdorff measure (see, for example,~\cite{Falconer:Book}) differs
from the usual normalization of the $i$-dimensional Lebesgue measure by a factor of $
\omega_i$.  So it might be useful to think in terms of the Hausdorff measures rather
than the $\intrinsic_i$, which are defined in terms of Lebesgue measures, especially since we
are considering the magnitude of fractal sets as well.

We are then led to speculate about the general behaviour of the magnitude.  It is possible to have varying degrees of optimism about this, so here are some possible conjectures.  The first one we suspect is too strong to be true.
\begin{StrongConjecture}
Let $\mathcal{C}$ be the class of all
compact subsets of $\R^m$, for any $m$.  Then there are unique functions $P,q\colon \mathcal{C}\to \R$  with the following properties:
\begin{enumerate}
\item $|A|=P(A)+q(A)$;
\item $P$ is a valuation, i.e.~satisfies the inclusion-exclusion principle;
\item if $A\in\mathcal{C}$ is polyconvex then $P(A) = \sum_{i \ge 0} \frac{1}{i!\,
\omega_i} \intrinsic_i(A)$;
\item for all $A\in\mathcal{C}$ we have $q(tA)\to 0$ as $t\to \infty$;
\item \label{Item:StrongConjConvex}
if $K\in \mathcal{C}$ is convex then $q(K)=0$.
\end{enumerate}
\end{StrongConjecture}
Perhaps the conjecture that we feel most plausible is essentially part~\ref{Item:StrongConjConvex} of the Strong Conjecture, namely that for any convex subset $K$ of Euclidean space, the magnitude is given by the valuation $P$, i.e.,
 \[|K|=\sum_{i\ge 0}\frac{ \intrinsic_i(K)}{i!\,\omega_i}.\]
We call this the Convex Magnitude Conjecture.  Further evidence for this kind of conjecture is presented in~\cite{Willerton:Heuristics, Meckes:PositiveDefinite, Leinster:Magnitude}.  A weaker conjecture is the following.
\begin{WeakConjecture}
There is a set ${\mathcal{D}}$ of
compact subsets of Euclidean spaces $\R^m$, which includes the set of finite sets of points, convex sets, circles and Cantor sets, and there is a unique function $P\colon {\mathcal{D}}\to \R$  with the following properties:
\begin{enumerate}
\item if $A \in {\mathcal{D}}$ then
$|tA| - P(tA) \to 0$ as $t\to\infty$;
\item $P$ is a valuation, i.e.~satisfies the inclusion-exclusion principle;
\item if $A\in {\mathcal{D}}$ is polyconvex then $P(A) = \sum_{i \ge 0} \frac{1}{i!\,
\omega_i} \intrinsic_i(A)$.
\end{enumerate}
\end{WeakConjecture}
A corollary of the Weak Asymptotic Conjecture would be the asymptotic inclusion-exclusion principle:
  \[|t(A\cup B)|-|tA|-|tB|+|t(A\cap B)|\to 0 \quad\text{as }t\to \infty\]
whenever $A,B,A\cup B, A\cap B\in {\mathcal{D}}$.

\subsection*{What is in this paper}
In the first section the notion of the magnitude of a finite metric space is motivated and defined. It is then proved that for any $n$-point metric space, provided that the points are sufficiently far
apart then the magnitude is well-defined and satisfies
\[\left|X\right|=n +q_1(X)\]
where $q_1(tX)\to 0$ as $t\to\infty$.  So asymptotically the magnitude is just the
cardinality of the underlying set or, equivalently, the Euler characteristic of the underlying topological space.  Note that for subsets of Euclidean space, the ``sufficiently far apart''
caveat is not necessary and the magnitude is always well-defined \cite{Leinster:Magnitude}.

In the second section we show that for a length $\ell$ closed line segment, $L_\ell$, we
can calculate
the magnitude using any sequence of finite subspaces of $\R$ converging in the
Hausdorff topology to the line segment and that the magnitude is
 \[\left|L_\ell\right|=\ell/2 +1.\]
This is \emph{precisely} half of the length plus the Euler characteristic, so there is no `asymptotic
correction'; this is perhaps related to the fact that a line segment is convex.

In the third section we consider $T_\ell$, the  `middle third' or `ternary' Cantor set of
length $\ell$.  We take an obvious sequence of approximations to this and calculate
the
magnitude of $T_\ell$ as the limit of the corresponding sequence of magnitudes.  We
find that this is of the form
\[\left|T_\ell\right|=p(\ell)+q_2(\ell)\]
where $q_2(\ell)\to 0 $ as $\ell\to\infty$ and $p$ satisfies the functional equation
$2p(\ell)=p(3\ell)$, which is an inclusion-exclusion result as the ternary Cantor set has
the self-similarity property $T_\ell\sqcup T_\ell=T_{3\ell}$.  This means that the
magnitude $|T_\ell|$ grows like $\log_32$, the Hausdorff dimension of the Cantor set.

In the fourth section the focus is on circles.  Taking $C_\ell$ to be the circle of
circumference $\ell$ embedded in Euclidean space in the canonical circular fashion,
and thus equipped with the subspace metric, we calculate the magnitudes of a sequence of symmetric approximation
to the circle, the limit of which
is the magnitude of the circle.  Using some non-trivial
classical asymptotic analysis we show that
  \[\left|C_\ell\right|= \ell/2 +q_3(\ell)\]
 where $q_3(\ell)\to 0$ as $\ell\to\infty$.  So again, asymptotically we get half of the length plus
the Euler characteristic, where here the Euler characteristic is zero.  We then go on to
use the same techniques to look at \emph{other} metrics on the circle.  The other
obvious metric on the circle is the `arc-length' metric; this can be viewed as an intrinsic
metric on the circle which does not depend on the embedding in Euclidean space, and
the magnitude of the circle with this metric is shown to have the same asymptotics.  In
fact there is a family of metrics interpolating and extrapolating the above two metrics; each of these metrics is obtained by embedding the circle in a constant curvature surface and using
the subspace metric.  We show that with these metrics we again have the same
asymptotic behaviour.

\subsection*{Acknowledgements}
We would like to thank Bruce Bartlett for his cunning use of \texttt{google} in the search for the proof of Theorem~\ref{Thm:EucCircleAsymptotics} and David Speyer for many comments, not least of which was his observation about homogeneous spaces (Theorem~\ref{Thm:SpeyersFormula}).    We would also like to thank the patrons of The $n$-Category Caf\'e for their contributions during the genesis of this paper (see~\cite{Leinster:MetricSpacesBlogPost}).  We thank Toby Bartels for some corrections to an early version, and Joe Fu for helpful comments.

Leinster is supported by an EPSRC Advanced Research Fellowship, and thanks the School of Mathematics and Statistics at the University of Sheffield for their generous hospitality.

\section{The magnitude of finite and positive definite metric spaces}

In the first part of this section
 a metric space is viewed as an enriched category so that the magnitude can be defined
analogously to the Euler characteristic of a finite category, introduced in~
\cite{Leinster:Cardinality} and~\cite{BergerLeinster:EulerCharDivergentSeries}.  This point of view acts purely as motivation, and an understanding of category theory is not necessary in order to read this paper.  After the definition, some basic properties are
given including a useful observation of David Speyer on the magnitude of homogeneous
metric spaces.  Next an overview of work of Mark Meckes on extending magnitude to a bigger class of metric spaces, namely positive definite spaces, is
given.  Finally in this section, it is shown that any finite metric space with sufficiently separated points has a well-defined magnitude, and that asymptotically the magnitude is simply the number of points.

\subsection{Category theoretic motivation}
 \label{Section:CategoryTheoreticMotivation}
 This part can be skipped if desired, but it describes where the seemingly \textit{ad hoc} definition of magnitude comes from.  Those desirous of some category theoretic background can refer to the books of Borceux and Mac Lane \cite{Borceux:Handbook1,Borceux:Handbook2,MacLane:Categories}.  
 
 Recall that a metric
space consists of a set $X$ together with a distance $d(x,x')\in [0,\infty)$ defined  for
each pair of elements $x,x'\in X$.  These are required to satisfy the triangle inequality
and the zero-self-distance axiom,
\[d(x,x')+d(x',x'')\ge d(x,x'')\quad\text{and}\quad 0\ge d(x,x)\]
(where the latter is usually written $0=d(x,x)$ but our reasons for writing it as we have
will become clear below), together with the symmetry and separation axioms,
\[d(x,x')=d(x',x)\quad\text{and}\quad 0\ge d(x,x') \Leftrightarrow x=x'.\]
Lawvere observed~\cite{Lawvere:MetricSpaces} that the first two of these conditions
are analogous to composition of morphisms and the inclusion of the identity in a
category:
\[\Hom(x,x')\times \Hom(x',x'')\to \Hom(x,x'') \quad\text{and}\quad \{\ast\}\to \Hom(x,x).
\]
Lawvere used this observation to interpret metric spaces as categories enriched over the
following monoidal category.  Take $[0,\infty]$ to be the monoidal category with the
non-negative real numbers together with infinity as its objects, and with precisely one
morphism from $a$ to $a'$ if $a\ge a'$ and no such morphisms otherwise; the
monoidal product is addition $+$ and the unit is $0$.  This category has categorical
products and coproducts: the categorical product of a set of objects is the supremum
of the set and the categorical coproduct is the infimum.

A category enriched over $[0,\infty]$ is then a set $X$ with a function $d\colon X\times
X\to [0,\infty]$  which satisfies the triangle inequality and zero-self-distance axiom.  (The
usual conditions these have to satisfy, namely associativity and the unit axiom, are
vacuously satisfied in this case.)  The notion of such an enriched category is then a
generalization of the usual notion of metric space in that the distance between two
points can be infinite, the distance does not have to be symmetric and the distance
between two different points can be zero; Lawvere argued convincingly that many
``metric spaces'' in nature are of this more general form.  In this paper the metric
spaces will be symmetric and satisfy the separation axiom, but the notion of magnitude
is defined in a similar way in the more general setting.

The definition of the Euler characteristic of a finite category can now be adapted to this
enriched category situation.  Recall briefly the definition of the Euler characteristic of a
finite category~\cite{Leinster:Cardinality}.  If $\mathcal{C}$ is a finite category then a \emph{weighting} on $
\mathcal{C}$ is a choice of real number $w_i\in\R$ for each object $i\in \text{Ob}
\mathcal{C}$ such that for every object $i$
 \[\sum_{j\in \text{Ob}\mathcal{C}}\Card\left(\Hom\left(i,j\right)\right)w_j=1\]
where $\Card$ is just the cardinality or number-of-elements function on finite sets.   If there exists a weighting on both $\mathcal{C}$ and $\mathcal{C}^{\text{op}}$ then the Euler characteristic is defined to be the sum of the
weights: $\chi(\mathcal{C}):=\sum_i w_i$.  The Euler characteristic is independent of the choice of weighting.

The key thing that needs adapting for the enriched case is the function~$\Card$.
We need a corresponding function $\widehat\Card\colon[0,\infty]\to\R$ on the objects
of the enriching category $[0,\infty]$.  The function $\Card$ satisfies $\Card(X\times Y)=
\Card(X)\times \Card(Y)$, so we will pick a function that satisfies $\widehat\Card(a+b)=
\widehat\Card(a)\widehat\Card(b)$.  The obvious choice for such a function is $
\widehat\Card(a)=\alpha^a$ for some non-negative number $\alpha$ and this is what we
shall use, taking $\alpha=e^{-1}$.
There is no obvious reason why other functions cannot be used;
it is just that our choice gives interesting results.  We can now give the definition.

\subsection{Definition and basic properties for finite metric spaces}
\label{Section:DefinitionOfMagnitude}
The magnitude of a finite metric space is defined in the following way.

Given a finite metric space $X$, a \emph{weighting} is a choice of real number $w_x\in\R$ for
each point $x\in X$ such that for each $x\in X$ we have
 \[\sum_{x'\in X}e^{-d(x,x')}w_{x'}=1.\]
If such a weighting exists then the \emph{magnitude} $|X|$ of the metric space $X$ is defined
to be the sum of the weights:
 \[|X|:=\sum_x w_x.\]

There are several things to note from this definition. Firstly, there are finite metric spaces on which no weighting exists (see \cite{Leinster:MetricSpacesBlogPost, Leinster:Magnitude})
 and the magnitude is then undefined.  Secondly, the weightings are not necessarily positive.  Thirdly, even if a weighting does exist then it might not be unique, however the magnitude is independent of the choice of weighting, as if $w$ and $\bar w$ are two weightings then
\[
\sum_x w_x
=
\sum_x w_x \sum_{x'} e^{-d(x, x')} \bar w_{x'}
=
\sum_{x'}\sum_{x} w_x e^{-d(x', x)}  \bar w_{x'}
=
\sum_{x'} \bar w_{x'},
\]
where we have used the symmetry of the metric.

In many cases, such as for positive definite spaces (see the next subsection),
the following gives a method for calculating the magnitude.  Define the
matrix $Z$ of exponentiated distances, indexed by the points of $X$, as follows:
$Z_{x,x'}:=e^{-d(x,x')}$.  \emph{If} $Z$ is invertible then $X$ has a unique weighting $w$, and  $w_{x}$ is the sum of entries in the $x$th
row in the inverse $Z^{-1}$:
\[w_x:=\sum_{x'}(Z^{-1})_{x,x'}\]
and thus the sum of all entries of $Z^{-1}$ gives the magnitude:
 \[|X|=\sum_{x,x'}(Z^{-1})_{x,x'}.\]
This method can be applied to `large' metric spaces in a sense made precise in Theorem~\ref{Theorem:SufficientlySeparated} and, in fact, to `most' metric spaces in a sense made precise in Proposition~2.2.6(i) of~\cite{Leinster:Magnitude}.

The following is a useful observation of David Speyer (stated in~\cite{Speyer:ncat}).
\begin{thm}[The Speyer Formula]\label{Thm:SpeyersFormula}
If a finite metric space $X$ carries a transitive action by a group of isometries then there
is a weighting in which the points all have the same weight and this is given on every
point by
  \[\frac{1}{\sum_{x'\in X}e^{-d(x,x')}}\]
for any $x\in X$.  Thus the magnitude is defined and is given by
\[|X|=\frac{\Card{X}}{\sum_{x'\in X}e^{-d(x,x')}}\]
for any $x\in X$.
\end{thm}
\begin{proof}
The observation is simply that as there is a transitive group action, for each $x\in X$ the
set-with-multiplicity of distances $\{d(x,x')\,| \,x'\in X\}$ is the same, so $\sum_{x'}e^{-d(x,x')}$ is also
the same for each $x$.  Thus for every $x$
  \[{\sum_{x'}\left(\frac1{\sum_{x''}e^{-d(x'',x')}}\right)e^{-d(x,x')}}
    =\frac{\sum_{x'}e^{-d(x,x')}}{\sum_{x''}e^{-d(x'',x')}}
    =1\]
and the weighting condition is satisfied.
\end{proof}
We can also define the \emph{magnitude function}.  For a finite metric space $X$ and a
positive number $t\in \R_{>0}$, let $tX$ be the metric space obtained by scaling all of
the distances by $t$, so that it has the same underlying set of points but the metric is given by $d_{tX}(x,x'):=td_{X}(x,x')$.  Now define the magnitude function, thought of as a function of $t$, to be
$|tX|$.  This is not necessarily defined for all $t$, but it \emph{is} defined for all $t$ sufficiently large, see Theorem~\ref{Theorem:SufficientlySeparated} below.
In this paper we will be interested in the asymptotics of this function for large $t$.


\subsection{The magnitude of a compact metric space}
\label{Section:MeckesResults}
In this subsection
we will just mention some work of Mark Meckes \cite{Meckes:PositiveDefinite} that followed on from an earlier version of this paper.  It transpires that a particularly nice class of spaces is that of positive definite spaces: a metric space $X$ is said to be 
\emph{positive definite} if for every finite subspace $W$, the matrix $Z\in \R^{W\times W}$ given by $Z(x,x'):= e^{-d(x,x')}$ is positive definite.  Finite positive definite spaces always have well-defined, non-negative magnitude.  Examples of positive definite spaces include subspaces of Euclidean space and round spheres; all of the infinite spaces considered in this  are of this nice form.

For a compact positive definite space $X$ the magnitude can be defined in the following two equivalent ways.  Firstly, we can define it as
  \[
  \left|X\right| :=\sup\left\{\left|W\right| \,\colon\, W\subset X \text{ with $W$ finite}\right\};
  \]
for non-positive definite finite metric spaces this will not necessarily agree with the usual definition of magnitude.  Secondly, by \cite[Corollary~2.7]{Meckes:PositiveDefinite}, if $(X_k)_{k\in \N}$ is a sequence of finite  subsets of $X$ with $X_k\to X$ in the Hausdorff metric then $\left|X_k\right|\to \left|X\right|$.  So this could equally be taken as a definition of magnitude.  (The definition of the Hausdorff metric is recalled in Section~\ref{Section:StraightLine}.)

Let $X$ be a compact positive definite metric space.  A \emph{weight measure} on $X$ is a finite signed Borel measure $\mu$ on $X$ such  that for all $x \in X$,
    \[
    \int_X e^{-d(x, x')} \, d\mu(x') = 1.
    \]
For any weight measure $\mu$ on $X$, we have $\mu(X) = |X|$ \cite[Theorem~2.3]{Meckes:PositiveDefinite}.  This gives a third possible definition of
magnitude, but one that only applies to spaces on which a weight
measure exists. 

Let $X$ be a compact metric space such that $tX$ is positive definite for all $t \gg 0$.  The \emph{magnitude dimension}~\cite[Definition~3.4.5]{Leinster:Magnitude}
of $X$ can be defined as the growth of the function $t \mapsto |tX|$:
\[
\dim(X)
=
\inf \Bigl\{
  \nu \in \mathbb{R} :
  \frac{|tX|}{t^\nu} \text{ is bounded for } t \gg 0
\Bigr\}.
\]
For example, it can be shown that if $A$ is a compact subset of $\R^n$ with the Euclidean metric then $\dim(A) \leq n$, with equality if $A$ has nonzero Lebesgue measure \cite[Theorem~3.5.8]{Leinster:Magnitude}.


\subsection{Asymptotics for finite metric spaces}
In this subsection it is shown that every finite metric space with sufficiently well separated points has a well-defined
magnitude.  Furthermore,  asymptotically it is just the number of points; if $X$ is a metric space with $n$ points, then in the
notation of the introduction, $\left | X\right|= n+q_1(X)$ where $q_1(tX)\to 0$ as $t\to\infty$.

Note that as subsets of Euclidean space are positive definite and finite positive definite spaces have a well-defined magnitude, any finite metric space
without a well-defined magnitude is not isometrically embeddable in Euclidean space; however the converse does not hold, as there are plenty of non-Euclidean metric spaces with a well-defined magnitude.

Consider first the two-point space $X_d$ where the two points are
a distance $d$ apart.
 \[\blob \stackrel {d}\longleftrightarrow \blob \]
 The magnitude $|X_d|$ is easy to calculate and there are many ways to do that.  As the
space is symmetric, Speyer's formula, Theorem~\ref{Thm:SpeyersFormula}, can be
applied to show that the magnitude is $2/(1+e^{-d})$, which clearly tends to $2$ as $d$
tends to infinity.  This can be rewritten as
   \[\left|X_d\right|=2 - \frac{2}{1+e^d}\]
which is in the form given above: it is the magnitude of the set of points plus
some term which is asymptotically zero as $d\to\infty$.

The general case of $n$ points requires a little analysis.  The first thing to do is to show
that a finite metric space has a well-defined magnitude provided that the points are
``sufficiently separated''.
\begin{thm}
\label{Theorem:SufficientlySeparated}
 If $X$ is a finite metric space with $n$ points such that the distance
between each pair of distinct points is greater than $\ln(n-1)$ then $X$ has a well-defined
magnitude.
\end{thm}
Note that whilst the bound only makes sense for $n\ge 2$, the condition is satisfied vacuously when $n$ is $0$ or $1$.

It was pointed out by the referee that this result is a consequence of the Levy-Desplanques Theorem~\cite[Theorem~6.1.10]{HornJohnson:MatrixAnalysis}, and that indeed the proof below can be generalized to give a proof of the Levy-Desplanques Theorem.
\begin{proof}
Firstly, $0$- and $1$-point spaces have magnitude $0$ and $1$ respectively, so we may assume that $n\ge 2$.
We wish to show that if $d(x,x')\ge \ln(n-1)$ for all $x\ne x'$ then the exponentiated distance
matrix $Z$ is invertible, so it suffices to show that if $Z$ is an $n \times n$ real matrix with $Z_{ii} =
1$ for all $i$ and $0 \leq Z_{ij} < 1/(n - 1)$ for all $i \neq j$ then $Z$ is
invertible.  In fact we show that $Z$ is positive definite:
${x}^\text{t} Z{x} \geq 0$ for all ${x} \in
\R^n$, with equality if and only if ${x} = 0$.  Indeed,
\begin{align*}
{x}^\text{t} Z {x}      &
=
\sum_i x_i^2 + \sum_{i \neq j} Z_{ij} x_i x_j
       \, \geq\,
\sum_i x_i^2 - \frac{1}{n - 1} \sum_{i \neq j} |x_i| |x_j|
        \\
        &=
\frac{1}{2(n - 1)} \sum_{i \neq j} (|x_i| - |x_j|)^2
        \,\geq\,
0.
\end{align*}
%
If ${x}^\text{t} Z{x} = 0$ then all of the inequalities must be equalities and so 
$|x_1| = \dots = |x_n| = \alpha$, say.  However, as there is the  
\emph{strict}
inequality $Z_{ij} < 1/(n - 1)$ it must be that $\alpha = 0$.  Thus $Z$ is positive definite and hence invertible, as required.
%
%
\end{proof}
By the above theorem, if $X$ is a finite metric space, then for $t\in \R$ sufficiently large, the scaled-up
version of $X$ has a well-defined magnitude $|tX|$, so it makes sense to talk of the
asymptotic behaviour of $|tX|$ even if the magnitude of $X$ itself is not defined.  The
fundamental result can now be stated.
\begin{thm}
\label{Theorem:AsymptoticFinitePoints}
If $X$ is a finite metric space with $n$ points then $|tX|\to n$ as $t\to \infty$.
\end{thm}
\begin{proof}
Let $\mathcal{M}_n$ be the space of real $n \times
n$ matrices.  On the subspace $\GL_n(\R)$ of \emph{invertible} matrices, define the
real-valued function $f$ by
\begin{align*}
f(Z) &{:=}
(\text{sum of the entries of } Z^{-1})\\
&=
(\text{sum of the entries of } \adj(Z))/\det(Z).
\end{align*}
where $\adj(Z)$ is the adjugate
matrix of $Z$.
Then $f(Z)$ is a rational function of the entries of $Z$, so $f$ is continuous.

Writing $Z_Y$ for the exponentiated distance matrix of a finite metric space $Y$, if $Z_{Y}$ is invertible then $\left|Y\right|=f(Z_{Y})$.  We have $\lim_{t \to
\infty} (Z_{tX}) = I$.  As $t\mapsto Z_{tX}$ is continuous, $I\in\GL_{n}(\R)$ and $\GL_n(\R)$ is
open in $\mathcal{M}_n$, then for $t\gg 0$ we have $Z_{tX} \in \GL_n(\R)$ and
\[
\lim_{t \to \infty} \left|tX\right|
=
\lim_{t \to \infty} f\left(Z_{tX}\right)
=
f\left(\lim_{t \to \infty} Z_{tX}\right)
=
f(I)
=
n.\qedhere
\]
%
\end{proof}
In other words, in the language of the introduction, if $X$ is finite metric space with $n$ points which has a magnitude then we may define $q_1(X)$ by
\[|X|=n+q_1(X),\]
and then, for $X$ any finite metric space,  $q_1(tX)\to 0$ as $t\to \infty$.


\section{The magnitude of a straight line segment}
\label{Section:StraightLine}
In this section we approximate $L_\ell$, a closed straight line segment of length $\ell$, by a sequence of finite metric spaces consisting of points lying in a line.  We show that
no matter which approximating sequence of this type is chosen, the sequence of magnitudes always converges to the same value, namely
$\left|L_\ell\right|$.  In fact
 $|L_\ell|=\ell/2+1$, which is \emph{exactly} the conjectured valuation; for this space there is no need to make an asymptotic statement.

Note that in~\cite{Willerton:Homogeneous}
the magnitude of a straight line segment is calculated using a weight measure.  This takes advantage of the fact that $L_\ell$ has a weight measure; we do not expect  $L_\ell^m$,  the $m$-dimensional cube of length $\ell$, to have a weight measure for any $m>1$.  The method given here is more general and demonstrates some elementary aspects of the theory.

Start by considering finite metric spaces consisting of points arranged in a line; we call these \emph{linear metric spaces}.  For
an $(n-1)$-tuple $\mathbf d =(d_1,\ldots,d_{n-1})$ of strictly positive real numbers,
define $X_{\mathbf d}$ to consist of $n$ points with the distance between
consecutive points being given by the $d_i$s, as in the following picture.
 \[X_\mathbf{d}:\qquad\blob \stackrel {d_1}\longleftrightarrow \blob \stackrel {d_2}\longleftrightarrow \blob
\dots\blob \stackrel {d_{n-1}}\longleftrightarrow \blob\]
This metric space has a weighting on it with the property that the weight of a point only depends on the distance to its nearest neighbours, giving rise to a simple expression for the magnitude.
\begin{thm}
\label{Thm:LinearMagnitude}
Suppose $X_{\mathbf{d}}$ is a linear metric space as above.  Then there is a weighting on it such that
weight of the $i$th point is
\[\frac 12 \left( \tanh \left(d_{i-1}/2\right) + \tanh \left(d_{i}/2\right)\right),\]
where, for convenience, we write $d_0=d_n=\infty$.
Thus the magnitude of the linear space $X_\mathbf{d}$ is given by:
  \[|X_\mathbf{d}|=1+\sum_{i=1}^{n-1}\tanh(d_i/2).
 \]
\end{thm}
\begin{proof}
The distance  $d_{ij}$ between the $i$th and $j$ points, for $i<j$, is given by $
\sum_{s=i}^{j-1} d_s$, so $e^{-d_{ij}}=\prod_s e^{-d_s}$.  Thus writing
$a_i:=e^{-d_i}$ the exponentiated distance matrix of $X_\mathbf{d}$ is given by the following matrix, which we write out for the case $n=5$, as the general pattern should be clear from this:
\[\begin{pmatrix} 1&a_1&a_1a_2&a_1a_2a_3&a_1a_2a_3a_4\\
a_1&1&a_2&a_2a_3&a_2a_3a_4\\
a_1a_2&a_2&1&a_3&a_3a_4\\
a_1a_2a_3&a_2a_3&a_3&1&a_4\\
a_1a_2a_3a_4&a_2a_3a_4&a_3a_4&a_4&1\\
\end{pmatrix}.
\]
It is easy to verify that the inverse of such a matrix is
\[\begin{pmatrix}
\frac1{1-a_1^2}& \frac{-a_1}{1-a_1^2}&0&0&0\\
\frac{-a_1}{1-a_1^2}&\frac{1+a^2_1}{2(1-a_1^2)}+\frac{1+a^2_2}{2(1-a_2^2)}&
\frac{-a_2}{1-a_2^2}&0&0\\
0&\frac{-a_2}{1-a_2^2}&\frac{1+a^2_2}{2(1-a_2^2)}+\frac{1+a^2_3}{2(1-a_3^2)}
&\frac{-a_3}{1-a_3^2}&0\\
0&0&\frac{-a_3}{1-a_3^2}&\frac{1+a^2_3}{2(1-a_3^2)}+\frac{1+a^2_4}{2(1-a_4^2)}
&\frac{-a_4}{1-a_4^2}\\
0&0&0&\frac{-a_4}{1-a_4^2}&\frac{1}{1-a_4^2}
\end{pmatrix}.\]
The weight of the $i$th point is the sum of the entries in the $i$th row, so for $i
\ne 1, n$ the weight is
\[\frac12 \left(\frac{1-a_{i-1}}{1+a_{i-1}}+\frac{1-a_{i}}{1+a_{i}}\right)\]
and as
\[\frac{1-a_{i}}{1+a_{i}}=\frac{a_i^{-1/2}-a^{1/2}_{i}}{a_i^{-1/2}+a^{1/2}_{i}}
=\tanh (d_i/2)\]
it follows that the weight is just $\frac12 \left(\tanh (d_{i-1}/2)+\tanh (d_{i}/2)\right)$.  At the
endpoints the weight is simply
 \[\frac1{1+a_m}=\frac12\left(\frac{1-a_m}{1+a_m}+1\right) = \frac12\left(\tanh (d_m/2)
+1\right)\]
 for $m=1,n-1$ respectively.
\end{proof}

We wish to approximate a straight line interval by a sequence of such finite linear spaces, where `approximate' means in the sense of the \emph{Hausdorff metric}.  Recall that the Hausdorff metric can be defined as follows (see, for example,~\cite{KlainRota:Book}).  If $X'$ is a compact subset of a metric space $X$ then for $\epsilon\ge0$, the \emph{$\epsilon$-expansion} $E(X',\epsilon)$ of $X'$ consists of all the points in $X$ of distance at most $\epsilon$ from a point in $X'$.  The Hausdorff distance between two compact subsets $X',X''\subseteq X$ is defined to be the least $\epsilon\ge0$ such that each subset is contained within the $\epsilon$-expansion of the other:
 \[d(X',X''):=\inf \{\epsilon\,|\, X'\subseteq E(X'',\epsilon) \text{ and } X''\subseteq E(X',\epsilon)\}.\]
 The following is straightforward from the definitions.
\begin{lemma}  Let $(X^k)_{k=1}^\infty$ be a sequence of finite subspaces of the length $\ell$ line segment $L_\ell$ so that  for each $k$ we have $X^k\cong X_{\mathbf{d}^k}$ for some tuple  $\mathbf {d}^k =(d_1,\ldots,d_{n_k-1})$.
Then $X^k\to L_{\ell}$ as $k\to \infty$ if and only if  $\sum_i d^k_i\to \ell $ and  $\max_i(d_i^k)
\to 0$ as $k\to \infty$.
\end{lemma}
The reader unfamiliar with the Hausdorff metric can take this as the definition of convergence.  We can now see that the limiting magnitude of such spaces is well-defined.
\begin{prop}
If $(X^k)_{k=1}^\infty$ is a sequence of finite subsets of $L_\ell$, a straight line segment of length $\ell$, which
converges to $L_\ell $ then
the sequence of magnitudes converges:
   \[X^k\to L_\ell\quad\text{as }k\to\infty\quad \Rightarrow\quad|X^k|\to \ell/2 +1\quad\text{as }k\to\infty.\]
\end{prop}
\begin{proof}
By the above lemma we can associate a sequence of tuples $(\mathbf{d}^k)$ such that  $\sum_i d^k_i\to \ell $ and  $\max_i(d_i^k)
\to 0$.
 Theorem~\ref{Thm:LinearMagnitude} implies $|X^k|=1+\sum_i \tanh (d_i^k/2)$. Thus it suffices to prove that $\sum_i \tanh (d_i^k/2)\to \ell/2 $ as $k\to \infty$.  This
requires a small amount of analysis.
%

We will first show that $\left| \tanh c -c\right|\le c^2$ for $c>0$.  By Lagrange's form for the remainder in the Taylor series,  there exists $\xi\in (0,c)$ such that
\[\tanh(c)=\tanh(0)+\tanh'(0)c+\tfrac12 {\tanh''(\xi)}c^2.\]
As $\tanh(0)=0$, $\tanh'(0)=1$ and $\tanh''(\xi)=2(\tanh^2(\xi)-1)\tanh(\xi)$ we get
 \[\left|\tanh(c)-c\right|=\left|(\tanh^2(\xi)-1)\tanh(\xi)\right|c^2<c^2,\]
as required; the last inequality due to the fact that $|\tanh(\xi)|<1$.

Now we can see
 \begin{align*}\textstyle
  \left|\sum_i \tanh (d_i^k/2) -\ell/2\right|
   &=\textstyle\left | \sum_i \tanh (d_i^k/2) - \sum_i d_i^k/2 +  \sum_i d_i^k/2 -\ell/2 \right |\\
   &\le\textstyle \left| \sum_i \left(\tanh (d_i^k/2) - d_i^k/2 \right)\right|
  + \frac12\left|\sum_i d_i^k -\ell \right|\\
  &\le\textstyle \sum_i (d^k_i/2)^2 +\tfrac12\left|\sum_i d_i^k -\ell \right|\\
  &\le \textstyle(\max_i d_i^k)\sum_id_i^k/4 +\frac12\left|\sum_i d_i^k -\ell \right|\\
  &\to 0\cdot \ell/4 +\tfrac12\cdot0 = 0 \quad \text{as }k\to \infty,
  \end{align*}
so $\sum_i \tanh (d_i^k/2)\to \ell/2 $ and $|X^k|\to \ell/2 +1$ as required.
\end{proof}
This proposition together with \cite[Corollary~2.7]{Meckes:PositiveDefinite} gives us the magnitude of any straight line segment.
\begin{thm}
The magnitude of the straight line segment of length $\ell $ has the following form.
\[|L_\ell |:=\ell/2 +1.\]
\end{thm}
This is precisely the conjectured form, even non-asymptotically, thus supporting part~(\ref{Item:StrongConjConvex}) of the Strong Asymptotic Conjecture in the introduction.

\section{The magnitude of a ternary cantor set}
Here we consider $T_\ell $, the ternary Cantor set of length $\ell $, where $\ell>0$.  We will calculate
the
magnitude $|T_\ell |$ of this Cantor set as a limit and show that the result is
consistent with the belief that asymptotically the magnitude satisfies the inclusion-exclusion principle.  A different approach to this calculation, using weight measures, is given in~\cite{Willerton:Homogeneous}. 

The ternary Cantor set $T_\ell $ is constructed by starting with a closed straight line segment of length $\ell $,
removing the open middle third, then removing the middle thirds of the the two
remaining components and continuing like this \textit{ad infinitum}.  This process
means that $T_{3\ell }$, the Cantor set of length $3\ell $, can be decomposed into two
copies of $T_\ell $, so
 \[T_{3\ell }=T_\ell \sqcup T_\ell .\]
Thus if $P$ is any valuation, i.e. satisfies the inclusion-exclusion principle, which is defined on
some collection of sets including these Cantor sets, then
\[P(T_{3\ell })=2P(T_\ell ).\]
Writing $p(\ell ):=P(T_\ell )$ we get the functional equation
\[p(3\ell )=2p(\ell ).\]
We can characterize the functions satisfying this functional equation in the following
way.
\begin{lemma}
\label{Lemma:FunctionalEquation}
Suppose $p\colon \R_{>0}\to \R$ is a function defined on the positive real numbers;
then $p$ satisfies the functional equation
\[p(3\ell )=2p(\ell )\quad\text{for all }\ell >0\]
if and only if $p$ is of the form
 \[p(\ell )=f(\ell )\ell ^{\log_3(2)}\]
where $f\colon \R_{>0}\to \R$ is some multiplicatively periodic function in the sense that
$f(3\ell )=f(\ell )$ for all $\ell >0$.
\end{lemma}
Whilst such  multiplicatively periodic functions are less frequently encountered than their additive counterparts, they are no less common: to obtain such a multiplicatively periodic function $f$, pick a
period-one ordinarily periodic function $g\colon \R\to \R$ and define
$f(\ell ):=g(\log_3(\ell ))$.
\begin{proof}
Suppose first that $p$ is a solution of the functional equation, then define
$f(\ell ):=p(\ell )\ell ^{-\log_3(2)}$ for $\ell >0$.  As $3^{\log_3(2)}=2$ we have
 \[f(3\ell )=p(3\ell )(3\ell )^{-\log_3(2)}=2p(\ell )2^{-1}\ell ^{-\log_3(2)}=f(\ell ),\]
and so $p$ has the required form.

Conversely, if $f$ is a function satisfying $f(3\ell )=f(\ell )
$, then defining $p(\ell ):=f(\ell )\ell ^{\log_3(2)}$ is easily seen to give a function satisfying
the functional equation:
\[p(3\ell )=f(3\ell )(3\ell )^{\log_3(2)}=f(\ell )2\ell ^{\log_3(2)}=2p(\ell ),\]
as required.
\end{proof}
The appearance of $\log_3(2)$ here is not outrageous as it is the Hausdorff dimension of
the Cantor set $T_\ell $ for every $\ell $.
We will show that the magnitude $|T_\ell |$ is of the form $p(\ell )+q_2(\ell )$ where $p$
satisfies the above functional equation and $q_2(\ell )\to 0$ as $\ell \to \infty$.  Hence the
magnitude dimension (see Section~\ref{Section:MeckesResults}) of the Cantor set is equal to its Hausdorff dimension.

 We will use a more constructive definition of the Cantor set $T_\ell $.  We start with the
zeroth approximation $\TC{\ell}{0}$ which consists of two points on the real line a  distance $\ell $ apart.   Let $\psi_{1}$ and $\psi_{2}$ be the two scalings of the real line by a factor of $1/3$ with the two points of $\TC{\ell}{0}$ as their respective fixed points.  Define $\TC\ell{k}$, the $k$th approximation to the Cantor set, inductively by $\TC\ell {k}:=\psi_{1}(\TC\ell {k-1})\cup\psi_{2}(\TC\ell {k-1})$ for $k\ge1$.  Then by the work of Hutchinson (see~\cite{Hutchinson:FractalsSelfSimilarity}) the length $\ell $
Cantor set is the limit of these sets,
$T_\ell =\bigcup_k \TC\ell{k}$, and it is the unique non-empty compact subset satisfying $T_\ell=\psi_{1}(T_\ell)\cup\psi_{2}(T_\ell )$.  Using the formula for the magnitude of a set of points in a
line given above, it is easy to calculate the magnitudes of these finite approximations to
the Cantor set.
\begin{thm}
The magnitude of the $k$th approximation to the Cantor set of length $\ell$ is
\[\left|\TC\ell {k}\right|=1+2^{k}\tanh\left(\frac{\ell }{2\cdot3^k}\right)+\frac12\sum_{i=1}^{k} 2^i
\tanh\left(\frac{\ell }{2\cdot3^i}\right).\]
\end{thm}
\begin{proof} We use the formula of Theorem~\ref{Thm:LinearMagnitude} for the magnitude of a linear metric space in terms of the distances between neighbouring points.   Since the $(k+1)$th approximation, $\TC\ell{k+1}$, is two copies of  $\TC{\ell/3}{k}$ a distance $\ell/3$ apart, we know that for a pair of neighbouring points in $\TC\ell{k+1}$, either both are in the same copy of  $\TC{\ell/3}{k} $  or else they are in different copies and are a distance $\ell/3$ apart.  So by Theorem~\ref{Thm:LinearMagnitude}, for $k\ge0$,
 \begin{align*}
 \left|\TC\ell{k+1}\right|
   &=
  1+2\left(\left|\TC{\ell/3}{k}\right| -1\right)+\tanh\left(\frac\ell {2\cdot3}\right).
\end{align*}
As we know that $\left|\TC\ell{0}\right|=1+\tanh(\ell/2)$, the result follows from a straightforward induction argument.
\end{proof}
We compute the magnitude of the Cantor set as the limit of these magnitudes.
\begin{thm}
The magnitude of $T_\ell$, the length $\ell$ ternary Cantor set, is given by
\[\left|T_\ell \right|=1+\frac12\sum_{i=1}^{\infty} 2^i\tanh\left(\frac{\ell }{2\cdot3^i}\right).\]
\end{thm}
\begin{proof}
 Since $\left|\tanh c\right| \le c$ for all $c \ge 0$, the sum $\sum_{i = 1}^\infty 2^i
\tanh \left( {\ell}/{2\cdot3^i} \right)$ converges and $2^{k} \tanh (
{\ell}/{2\cdot3^k} ) \to 0$ as $k \to \infty$.  Hence $\lim_{k \to
\infty} |\TC\ell{k}|$ exists.  We know by \cite[Corollary~2.7]{Meckes:PositiveDefinite} that the magnitude of the length $\ell$ Cantor set is this limit, and so the result follows.
\end{proof}
This magnitude can be decomposed as promised above. First define
\[p(\ell ):=\frac12\sum_{i=-\infty}^{\infty} 2^i\tanh\left(\frac{\ell }{2\cdot3^i}\right);\quad
q_2(\ell ):=1-\frac12\sum_{i=0}^{\infty} \frac1{2^i}\tanh\left(\frac{ 3^i\ell}{2}\right).\]
Note the doubly infinite summation in the definition of $p$: the part of the sum indexed by negative $i$ and the sum in the definition of $q_2$ both converge because $\tanh$ is bounded.

The promised result is almost immediate.
\begin{thm}
\label{Thm:CantorMagDecomposition}
The magnitude of the length $\ell$ Cantor set decomposes uniquely as a piece satisfying the functional equation and a piece which is asymptotically zero:
\enlargethispage*{1em}
\begin{enumerate}
\item $|T_\ell |=p(\ell )+q_2(\ell )$;
\item $p(3\ell )=2p(\ell )$;
\item $q_2(\ell )\to 0$ as $\ell \to \infty$;
\item the functions $p$ and $q_2$ are uniquely determined by the above three properties.
\end{enumerate}
\end{thm}
\begin{proof}
\begin{enumerate}
\item This follows by definition.
\item This follows immediately by substitution.
\item This follows from the fact that $\tanh d\to 1$ as $d\to \infty$.
\item Suppose $\tilde p$ and $\tilde q_{2}$ form another decomposition of $|T_{\ell}|$ with these properties.  Then defining $s:=p-\tilde p=\tilde q_{2}-q_{2}$ gives a function which satisfies both $s(3\ell)=2s(\ell)$ and $s(\ell)\to 0$ as $\ell\to \infty$, from which it follows that $s$ is identically zero, thus $\tilde p=p$ and $\tilde q_{2}=q_{2}$. \qedhere
\end{enumerate}
\end{proof}
The functions $|T_\ell |$ and $p(\ell )$ are plotted in Figure~\ref{figure:CantorCard}.
\begin{figure}
%
%
%
%
\begin{center}
\begin{tikzpicture}
\begin{axis}[
axis equal image=true,
axis x line=bottom, axis y line = left,
xmin=0,ymin=0, ymax=5, xmax=10,
xtick={0,10}, xticklabels={0,10},
ytick={0,1,5}, yticklabels={0,1,5},
x axis line style={style = -},y axis line style={style = -},
xlabel=$\ell$,
legend style={at={(1,0.1)},anchor=south east}
]
\addplot[mark=none] file {CantorMagGraph.table.tex};
\addplot [mark=none,dashed] file {CantorMagValuationGraph.table.tex};
\legend{$|T_\ell|$,$p(\ell)$};
\end{axis}

\end{tikzpicture}
\end{center}
\caption{The magnitude $|T_\ell|$ of the length $\ell$ Cantor set is asymptotically the same as the valuation-like function $p(\ell)$, and the latter satisfies $1.205\;\ell ^{\log_3(2)}< p(\ell )<1.206\;\ell ^{\log_3(2)}$ (according to \texttt{maple}).}
\label{figure:CantorCard}
\end{figure}
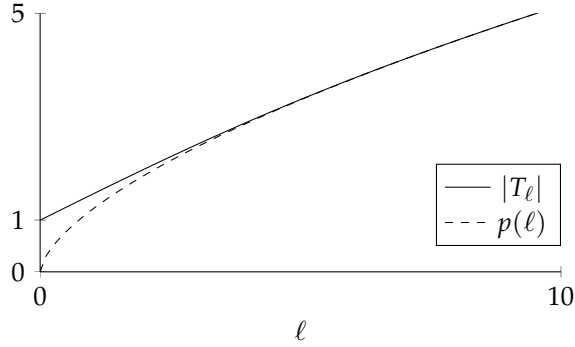

We can now get $p(\ell )$ into the form $f(\ell)\ell^{\log_{3}2}$ where $f$ is multiplicatively periodic as in Lemma~\ref{Lemma:FunctionalEquation}.  Define
 \[f(\ell ):=\ell ^{-\log_3(2)}\frac12\sum_{i=-\infty}^{\infty} 2^i\tanh\left(\frac{\ell }{2\cdot3^i}\right)
\]
so that $p(\ell )=f(\ell )\ell ^{\log_3(2)}$.  The multiplicative periodicity of $f$ is clear when it
is written in terms of the base-three logarithm of $\ell $:
 \[f(3^x)=\frac12\sum_{i=-\infty}^{\infty} 2^{i-x}\tanh\left(3^{x-i}/2\right).\]
%
%
It can be seen numerically that $f$ is not far from being a constant function; indeed, using \texttt{maple}, we can calculate the Fourier expansion which we see has rapidly decaying coefficients:
\[f(3^x)\simeq1.2054 + 2.48\times10^{-4}\sin(2\pi x + \theta_1)+3.36\times10^{-8}\sin(4\pi x+\theta_2)+\dots, \]
where $\theta_1$ and $\theta_2$ are some constants.  Intriguingly, similar near-constant functions arising from the same functional equation were studied in~\cite{BigginsBingham:NearConstancy}.  From these numerics we get the following bounds for the valuation-like function $p$:
\[1.205\;\ell ^{\log_3(2)}< p(\ell )<1.206\;\ell ^{\log_3(2)}.\]

In conclusion, Theorem~\ref{Thm:CantorMagDecomposition} tells us that although the magnitude of the Cantor set $|T_{\ell}|$ does \emph{not} satisfy the functional equation, so in general $|T_{3\ell}|\ne 2|T_{\ell}|$, it \emph{is} asymptotically equal to a function $p(\ell)$ that does satisfy the functional equation.  Moreover, $|T_{\ell}|$ has growth $\log_{3}2$, which is the Hausdorff dimension of the Cantor set.

\section{The magnitude of a circle}
In this section we consider the magnitude of circles.  There are actually several metrics that can be put on a circle.  In this paper we are primarily interested in subsets of Euclidean space with the subspace metric, so we will first consider the subspace metric on circles which are embedded in the natural `round' way in $\R^2$.  We will see that the magnitude of $C_\ell$, a circle of circumference $\ell$, is of the form $\left|C_\ell\right|= \ell/2 +q_3(\ell)$, where $q_3(\ell)\to 0$ as $\ell\to \infty$.   This requires some non-trivial asymptotic analysis.
After doing this we will move away from subspaces of Euclidean space and show that the same result holds for other natural metrics on the circle.  Despite considering non-Euclidean supspaces,  we can apply the same techniques since, by Theorem~3.6(6) of~\cite{Meckes:PositiveDefinite}, these metrics are positive definite.  As in~\cite{Willerton:Homogeneous}, the magnitude calculations performed here can also be done with weight measures.

\subsection{The subspace metric on the circle}
\label{Subsection:EucMetricOnCircle}
In this section we calculate the magnitude of a circle with the subspace-of-Euclidean-space metric, and then use Laplace's method from asymptotic analysis to show that as the length increases the magnitude becomes close to half the length.

First consider a circle $C_\ell$ of length, or circumference, $\ell$ as a subset of $\R^2$ with the induced metric.  This means that the distance between points $p_1$ and $p_2$ on the circle which subtend an angle $\theta$ at the origin is given by
 \[d(p_1,p_2)=\tfrac \ell \pi \sin \tfrac \theta 2,\]
as can be seen from the following picture.
  \[\raisebox{-.45\height}{\begin{picture}(0,0)%
\includegraphics{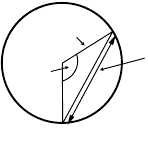}%
\end{picture}%
\setlength{\unitlength}{3947sp}%
\begingroup\makeatletter\ifx\SetFigFont\undefined%
\gdef\SetFigFont#1#2{%
  \fontsize{#1}{#2pt}%
  \selectfont}%
\fi\endgroup%
\begin{picture}(1268,1125)(256,-392)
\put(865,438){\makebox(0,0)[rb]{\smash{{\SetFigFont{8}{9.6}{\color[rgb]{0,0,0}$\tfrac \ell{2\pi}$}%
}}}}
\put(642,146){\makebox(0,0)[rb]{\smash{{\SetFigFont{8}{9.6}{\color[rgb]{0,0,0}$\theta$}%
}}}}
\put(1420,229){\makebox(0,0)[lb]{\smash{{\SetFigFont{8}{9.6}{\color[rgb]{0,0,0}$\tfrac{\ell}{\pi}\sin\tfrac{\theta}{2}$}%
}}}}
\put(1217,489){\makebox(0,0)[lb]{\smash{{\SetFigFont{8}{9.6}{\color[rgb]{0,0,0}$p_1$}%
}}}}
\put(755,-356){\makebox(0,0)[b]{\smash{{\SetFigFont{8}{9.6}{\color[rgb]{0,0,0}$p_2$}%
}}}}
\end{picture}%
}\]

 Now we will approximate the circle by a finite set of points.  We define $K^n_\ell$ to be a set of $n$ points equally spaced around the circle, equipped with the subspace metric.  This finite metric space is homogeneous, as it carries a transitive group action of the cyclic group of order $n$, so we can apply Speyer's Formula (Theorem~\ref{Thm:SpeyersFormula}) to see that the magnitude of this finite approximation to the circle is given by
  \[\left|K_\ell^n\right|=
 \frac n{\sum_{j=1}^n \exp\left(\frac{-\ell}{\pi} \sin\frac{\pi j}{n}\right)}
 =
 \frac 1{\sum_{j=1}^n \exp\left(\frac{-\ell}{\pi} \sin\frac{\pi j}{n}\right)
 \frac1n}.
 \]
 We can take the limit as the number of points tends to infinity and see that the denominator just consists of Riemann sums, so tends to an integral:
 \[\left | K_{\ell}^n\right|
             \to \frac 1{\int_{0}^1 \exp\left(\frac{-\ell}{\pi} \sin(\pi s)\right)\text{d}s} \quad\text{as }n\to\infty.\]
As these subsets converge to the circle, we know, by~\cite[Corollary~2.7]{Meckes:PositiveDefinite}, that the magnitude of the circle is just this limit; we have therefore proved the following.
\begin{thm}
For $C_\ell$ the length $\ell$ circle equipped with the Euclidean subspace metric, the magnitude is given by
\[|C_{\ell}|=\left(\int_{0}^1 \exp\left(-\tfrac{\ell}{\pi} \sin(\pi s)\right)\text{d}s\right)^{-1}.\]
\end{thm}
This magnitude is plotted for some values in Figure~\ref{Figure:MagnitudeSubspaceCircle}, and it is seen that the magnitude appears to be approaching half of the length; some classical asymptotic analysis shows that this is indeed the case.
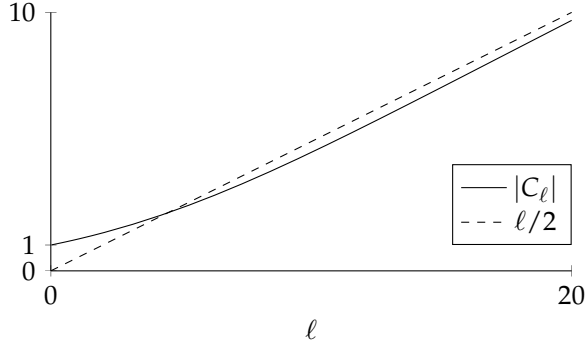
\begin{figure}
\begin{center}
%
%
%
%
\begin{tikzpicture}
\begin{axis}[
axis equal image=true,
axis x line=bottom, axis y line = left,
xmin=0,ymin=0, ymax=10, xmax=20,
xlabel=$\ell$,
xtick={0,20}, xticklabels={0,20},
ytick={0,1,10}, yticklabels={0,1,10},
x axis line style={style = -},y axis line style={style = -},
legend style={at={(1,0.1)},anchor=south east}
]
\addplot[mark=none] file {CircleMagGraphK0.table.tex};
\addplot [mark=none,dashed] coordinates {(0,0) (20,10)};
\legend{$|C_\ell|$,$\ell/2$};
\end{axis}

\end{tikzpicture}

\end{center}
\caption{The magnitude $|C_\ell|$ of the length $\ell$ circle with the subspace metric is asymptotically the same as half of the length.}
\label{Figure:MagnitudeSubspaceCircle}
\end{figure}
\begin{thm}\label{Thm:EucCircleAsymptotics}
For $C_\ell$ the length $\ell$ circle equipped with the Euclidean subspace metric, the magnitude satisfies
 \[\left|C_\ell\right|=\ell/2+q_3(\ell),\]
where $q_3(\ell)\to 0$ as $\ell\to \infty$.
\end{thm}
\begin{proof} We will in fact prove that $\left|C_\ell\right|=\ell/2+O(\ell^{-1})$ as $\ell\to \infty$; this follows by taking the reciprocal of both sides of the assertion
 \[\int_0^1 e^{-\frac{\ell }{\pi}\sin(\pi s)}\text{d}s = \frac 2{\ell }+ O(\ell^{-3})
     \qquad \text{as }\ell\to\infty,\]
 which is what we will now derive using the classical asymptotic analysis technique known as Laplace's method (see, for example,~\cite{Miller:AppliedAsymptoticAnalysis}).

Define $D_0(s):=\frac{1}{\pi}\sin(\pi s)$ for $s\in[0,1]$; this is plotted in Figure~\ref{fig:DistGraph}.  Thus $D_0$ is a function that
takes minimum value zero precisely at its endpoints, $D_0(0)=0=D_0(1)$, and is infinitely differentiable
 with $D_0'(0),D_0'(1)\ne 0$.
 We wish to consider the asymptotic behaviour, as $\ell\to\infty$, of the integral
 \[F(\ell):=\int_0^1e^{-\ell D_0(s)}\text{d}s.\]
 We break the integral into two pieces:
 \begin{align*}
 	F(\ell)&=\int_0^{1/2} e^{-\ell D_0(s)}\text{d}s
             +\int_{1/2}^1e^{-\ell D_0(s)}\text{d}s\\
            &=:\qquad F_\text{left}(\ell ) \qquad+\qquad F_\text{right}(\ell ).
 \end{align*}
Concentrating on the asymptotics of $F_{\text{left}}$, as $D_0$ is one-to-one on $[0,1/2]$, it has an inverse there, so we can use a change of variable and Watson's Lemma (see~\cite[Chapter 3.3]{Miller:AppliedAsymptoticAnalysis} for the details) to obtain
  \[F_\text{left}(\ell )=\frac{1}{D_0'(0)\ell } -\frac{D_0''(0)}{D_0'(0)^3 \ell ^2} +
  O(\ell ^{-3}) \qquad\text{as }\ell\to\infty.
  \]
As  $D_0'(0)=1$ and $D_0''(0)=0$,
we obtain
  \[F_\text{left}(\ell)=\frac{1}{\ell }
  +O(\ell ^{-3}).
  \]
By symmetry $F_{\text{right}}$ gives the same contribution; adding these together gives the asymptotics of the integral:
 \[\int_0^1 e^{-\frac{\ell }{\pi}\sin(\pi s)}\text{d}x =\frac{2}{\ell}
 +O(\ell ^{-3})\qquad\text{as }\ell\to\infty
 \]
which suffices to prove the theorem.
\end{proof}
The analysis in the proof can be easily extended to show that the magnitude function
asymptotically looks like $\ell/2 -\pi^2/2\ell +\dots$.

Similarly, one can try to do a Taylor expansion of the magnitude function around $\ell=0$.  For instance,
the derivative of $|C_{\ell }|$ at $\ell=0$ is $\frac{\ell}{\pi} \int^1_0 \sin(\pi s)\text{d}s={2\ell}/{\pi^2}$.  It is
not clear if the $\pi^2/2$ there is related to the
one in the preceding paragraph.

\subsection{The intrinsic metric}
Another obvious, perhaps more obvious, choice of metric on a length $\ell$ circle is the `arc-length' metric so that the distance between two points $p_1$ and $p_2$
on the circle which subtend an angle $\theta\in[0,\pi]$ at the centre is $\ell\theta/2\pi$.
 \[\raisebox{-.45\height}{\begin{picture}(0,0)%
\includegraphics{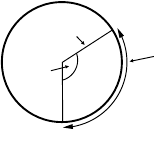}%
\end{picture}%
\setlength{\unitlength}{3947sp}%
\begingroup\makeatletter\ifx\SetFigFont\undefined%
\gdef\SetFigFont#1#2{%
  \fontsize{#1}{#2pt}%
  \selectfont}%
\fi\endgroup%
\begin{picture}(1286,1133)(256,-400)
\put(1235,526){\makebox(0,0)[lb]{\smash{{\SetFigFont{8}{9.6}{\color[rgb]{0,0,0}$p_1$}%
}}}}
\put(642,146){\makebox(0,0)[rb]{\smash{{\SetFigFont{8}{9.6}{\color[rgb]{0,0,0}$\theta$}%
}}}}
\put(865,438){\makebox(0,0)[rb]{\smash{{\SetFigFont{8}{9.6}{\color[rgb]{0,0,0}$\tfrac \ell{2\pi}$}%
}}}}
\put(1516,241){\makebox(0,0)[lb]{\smash{{\SetFigFont{8}{9.6}{\color[rgb]{0,0,0}$\tfrac{\ell \theta}{2\pi}$}%
}}}}
\put(758,-364){\makebox(0,0)[b]{\smash{{\SetFigFont{8}{9.6}{\color[rgb]{0,0,0}$p_2$}%
}}}}
\end{picture}%
}\]
 We will denote the length $\ell$ circle with this metric by $\bar C_\ell$.  On the one hand this metric can be viewed as `the' intrinsic metric on the circle which exists independent of any embedding and corresponds to the natural Riemannian structure on the circle.  On the other hand this can be seen as another way to put a `subspace metric' on a subset of a metric space: on a metric space one can define the arc-length of a path in the space and define a metric on any subspace by taking the distance between two points to be the length of the shortest path between them which lies in the subspace.

To calculate the magnitude of this metric space we again approximate the circle by a sequence of subsets of evenly spaced points.  Note that although this is not a metric inherited from Euclidean space, the metric is still positive definite~\cite[Theorem~3.6(6)]{Meckes:PositiveDefinite} and so the approximation approach still works.
So, let $\bar K^n_\ell$ be the metric space with $n$ points labelled $1,\dots,n$ such that the distance between the $i$th and $j$th points is given by
  \[d_{i,j}=\frac \ell n\min\left(|i-j|,n-|i-j|\right).\]
Again this is a homogeneous space as the cyclic group of order $n$ acts transitively on it, so we can apply Speyer's Formula (Theorem~\ref{Thm:SpeyersFormula}) to obtain the magnitude of this finite approximation
to the length $\ell $ circle with the intrinsic metric:
\[ \left|\bar K^n_\ell \right|=\frac n{\sum_{j=1}^{n}e^{-d_{1j}}}.\]
Now we can find the magnitude of the actual circle by letting the number of points tend to infinity whilst keeping the length fixed.
\begin{thm}
For $\bar C_\ell$ the length $\ell$ circle with its intrinsic metric, the magnitude is given by
  \[\left|\bar C_\ell \right| = \frac {\ell/2} {1-e^{-\ell/2 }}.\]
\end{thm}
\begin{proof}  We know by ~\cite[Corollary~2.7]{Meckes:PositiveDefinite}  that $\left|\bar K^n_\ell \right| \to \left|\bar C_\ell \right|$, so we just need to prove $\left|\bar K^n_\ell \right| \to \frac {\ell/2} {1-e^{-\ell/2 }}$.
There are at least two ways to prove this.  One can express the denominator in the formula for $|\bar K^n_\ell|$ as a
geometric progression and then take the limit, or one can express the limit of the
denominator as an integral.  The second method is the analogue of the method in the last subsection and
so that is the method given here.

As the limit of the sequence of magnitudes $|\bar K^1_\ell|,|\bar K^2_\ell|,|\bar K^3_\ell|,\dots$ 
exists it is equal to the limit of the `odd' subsequence of magnitudes  $|\bar K^1_\ell|,|\bar K^3_\ell|,|\bar K^5_\ell|,\dots$
We have
\[\left|\bar K_\ell ^{2m+1}\right|=\frac {2m+1}{2\sum_{k=1}^{m}e^{-\ell k/(2m+1)} +1}
=\frac {\ell }{2\sum_{k=1}^{m}e^{-\ell k/(2m+1)}\frac \ell {2m+1} +\frac \ell {2m+1}}.\]
Now observe that
\[\sum_{k=1}^{m}e^{-\ell k/(2m+1)}\frac \ell {2m+1}\longrightarrow \int^{\ell /2}_0
e^{-x} dx
= 1-e^{-\ell/2 }\quad\text{as }m\to\infty.\]
Thus
 \[\left|\bar K_\ell ^{2m+1}\right| \to \frac {\ell/2 }{1-e^{-\ell/2 }}\quad\text{as }m\to \infty,\]
and the proof is
completed.
\end{proof}
In other words we have
\[|\bar C_\ell | =\ \ell/2+\frac{\ell/2}{e^{\ell/2}-1},\]
and we immediately see that this has the suggested behaviour of being half of the length plus an asymptotically zero piece.  This case is even better behaved than the subspace metric case considered above, in that this required far less analysis to derive and the difference between the magnitude and half of the length is exponentially small in the length --- the magnitude and half of the length are asymptotically the same to all orders.

\subsection{Other round metrics on the circle}
As with the last subsection, this subsection is slightly away from the main theme of the paper in that it deals with metric spaces which are not subspaces of Euclidean space, but are subspaces of manifolds with the subspace metric; however, these are still positive definite spaces and this should be of some interest anyway.  In the above two subsections we saw that there are two obvious metrics on the circle, the Euclidean subspace metric and the intrinsic metric.  Here we will see that they are part of a family of metrics on the circle; each member of the family is a subspace metric got by considering the circle of length $\ell$ as the locus of points equidistant from some fixed point on a homogeneous surface of specified curvature.  Think, for instance, of a circle of fixed latitude on the Earth as points equidistant from the North Pole.  In particular the intrinsic metric on a circle comes from embedding it as an equator in a sphere.  We will see that the asymptotic relationship to half of the length of the circle holds in an appropriate sense.

As shown in Figure~\ref{fig:SphericalMetric} a circle of length $\ell$ can be embedded `roundly',%
\footnote{This terminology is just to distinguish a `round' circle from a topological circle.}
 by which we will mean as a locus of points equidistant from some fixed point, on any sphere of radius at least $\ell/2\pi$.  We will consider the plane to be a sphere of infinite radius and as being the limit of arbitrarily large spheres.  The curvature of a sphere of radius $R$ is $1/R^2$, so a length $\ell$ circle  can be roundly embedded in any sphere of curvature from $(2\pi/\ell)^2$ down to $0$; and it will be convenient to define the \emph{relative curvature} of a radius $R$ sphere to be
    \[\kappa:=\frac{1/R^2}{1/(\ell /2\pi)^2}=\frac {\ell ^2}{4\pi^2 R^2}.\]
Thus a circle of length $\ell$ can be embedded in spheres of relative curvature from $1$ down to $0$.  We can equip the length $\ell$ circle  with the subspace metric, call this the $\kappa$-metric and denote the resulting metric space by $C_{\ell,\kappa}$.  Thus $C_{\ell,1}$ is just $\bar C_{\ell}$, the circle with the intrinsic metric, and $C_{\ell,0}$ is just $C_{\ell}$, the circle with the Euclidean metric.

We will now give a formula for the $\kappa$-metric.  First it is necessary to decide how to parametrize the length $\ell$ circle; we will parametrize it by arc-length, whereas in the previous two subsections we have parametrized by angle subtended at the centre.  Define $D_{\ell,\kappa}\colon [0,\ell]\to \R$ by taking $D_{\ell,\kappa}(x)$   to be the subspace distance between two points a distance $x$ away on the length $l$ circle which is roundly embedded in the sphere (or plane) of relative curvature $\kappa$, as in Figure~\ref{fig:SphericalMetric}.   In the case $\kappa=0$, by the formula in Section~\ref{Subsection:EucMetricOnCircle}, we have
\[D_{\ell ,0}(x)=\frac \ell {\pi} \sin \frac{\pi x}
{\ell },\]
whereas for $0<\kappa\le 1$ we have the following.
\begin{thm}\label{Thm:FormulaForMetricOnSphere}
The subspace metric on the length $\ell $ circle embedded roundly in
the $2$-sphere of radius $R$ and of relative curvature $\kappa=\ell ^2/4\pi^2 R^2>0$ is given by
\[D_{\ell ,\kappa}(x)=\frac \ell {\sqrt{\kappa}\pi}\sin^{-1}\left(\sqrt{\kappa} \sin \frac{\pi x}
{\ell }\right),\]
where $\sin^{-1}$ is interpreted as a function $[0,1]\to[0,\pi/2]$.
\end{thm}
\begin{proof}
The essential ideas of the proof are contained in Figure~\ref{fig:SphericalMetric}.
Suppose that we have a round, length $\ell $ circle on a sphere of radius $R$.  We will
make things easier for ourselves by embedding the $2$-sphere in the standard way into
Euclidean $3$-space as in Figure~\ref{fig:SphericalMetric}.  Suppose further that $P$
and $Q$ are two points on the circle such that the arc of the circle between them has
length $x$.  We wish to find $D_{\ell ,\kappa}(x)$ which is the spherical distance
between them, that is to say the the length of the arc of a great circle from $P$ to $Q$,
where a great circle is a circle on the sphere whose centre is at the origin.

\begin{figure}
\[\raisebox{-.45\height}{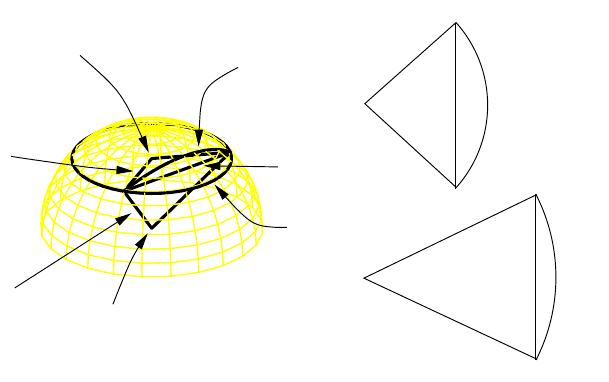}\]
\caption{Calculating the spherical distance between points $P$ and $Q$ on a length $
\ell $ circle embedded on a radius $R$ sphere: $x$ is the circular distance between the
points; $z$ is the Euclidean distance between the points; and $D_{\ell ,\kappa}(x)$ is
the spherical distance between the points.}
\label{fig:SphericalMetric}
\end{figure}

The points $P$ and $Q$ are the endpoints of segments of two circles, the circle of length
$\ell $ and the great circle passing through $P$ and $Q$.  These two segments are
pictured in Figure~\ref{fig:SphericalMetric} with $\theta$ written for the angle
subtended by $P$ and $Q$ at the centre of the length $\ell $ circle, $\phi$ written for
the angle subtended by $P$ and $Q$ at the centre of the sphere, and $z$ written for
the Euclidean distance between $P$ and $Q$.  From the two segments it is seen that
\[z=\frac \ell \pi \sin\frac\theta 2 \quad\text{and}\quad z=2R\sin\frac\phi 2,\]
but the angles $\theta$ and $\phi$ are seen to be given by
\[\theta=\frac {2\pi x}{\ell }\quad\text{and}\quad\phi=\frac{D_{\ell ,\kappa}(x)}{R},\]
so equating the two expressions for $z$ we get
\[\frac \ell \pi \sin\frac {\pi x}{\ell }=2R\sin\frac{D_{\ell ,\kappa}(x)}{2R}\]
whence
\[D_{\ell ,\kappa}(x)=2R\sin^{-1}\left(\frac{\ell }{2\pi R} \sin \frac{\pi x}{\ell }\right)
=\frac \ell {\sqrt{\kappa}\pi}\sin^{-1}\left(\sqrt{\kappa} \sin \frac{\pi x}{\ell }\right),\]
which is what was required.
\end{proof}

 Note that when $\kappa=1$ this recovers the intrinsic metric  and that
$D_{\ell ,\kappa}\to D_{\ell ,0}$ as $\kappa\to 0$; this corresponds to the idea that locally a large sphere
looks metrically like a patch of the Euclidean plane.  Thus this gives a family of metrics interpolating between the
Euclidean and intrinsic ones.

 This family of metrics can actually be extended by taking the relative curvature $
\kappa$ to be in $(-\infty,1]$.  The formula for $D_{\ell ,\kappa}$ is valid for negative $
\kappa$ and corresponds to a metric on the length $\ell $ circle induced by embedding
it roundly in a suitably curved hyperbolic space; this can be proved either by geometric
means akin to those in Theorem~\ref{Thm:FormulaForMetricOnSphere}, or by the standard
algebraic trick of considering a hyperbolic plane as a sphere of imaginary radius.  If the
reader is unhappy with the imaginary quantities in the expression for $D_{\ell ,\kappa}$
when $\kappa$ is negative then they could perhaps be reassured by rewriting the
expression for $D_{\ell ,\kappa}$ in terms of hyperbolic sines:
  \[D_{\ell ,\kappa}(x)=\frac{\ell }{\sqrt{-\kappa}\pi}\sinh^{-1}\left(\sqrt{-\kappa}\sin
\frac{\pi x}\ell \right )\quad\text{for }x\in [0,\ell ].\]

 Observe that $D_{\ell ,\kappa}$ has the scaling property
 \[D_{\ell ,\kappa}(\ell s)=\ell D_{1,\kappa}(s)\quad\text{for }s\in[0,1],\]
which is why we chose to parametrize by the relative curvature $\kappa$.
This means that we can concentrate on $D_{1,\kappa}$ which we will write as $D_\kappa$.



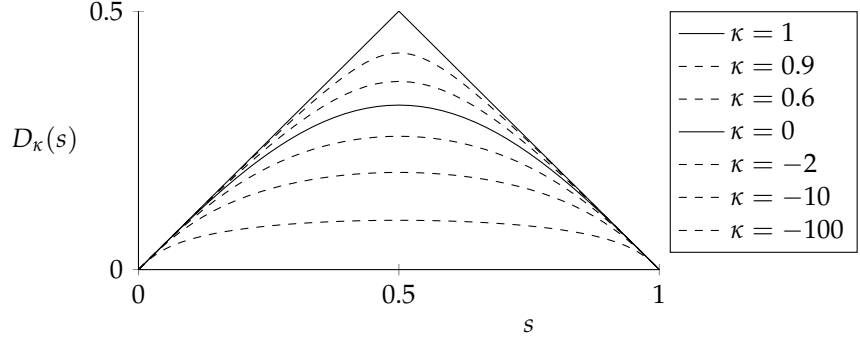
\begin{figure}
\begin{center}
%
%
%
%
\begin{tikzpicture}
\begin{axis}[
axis equal image=true,
axis x line=bottom, axis y line = left,
xmin=0,ymin=0, ymax=0.5, xmax=1,
xlabel=$s$, ylabel=$D_\kappa(s)$,
y label style = {rotate=-90}, x label style={at={(0.75,0)}},
xtick={0,0.5,1},
ytick={0,0.5}, 
x axis line style={style = -},y axis line style={style = -},
legend style={at={(1.02,1)},anchor=north west,cells={anchor=west}}
]
\addplot [mark=none] coordinates {(0,0) (0.5,0.5) (1,0)};
\addplot [mark=none,dashed]  file {CircleDistanceGraphkappa0.9.dat.tex};
\addplot [mark=none,dashed] file {CircleDistanceGraphkappa0.6.dat.tex};
\addplot [mark=none] file {CircleDistanceGraphkappa0.dat.tex};
\addplot [mark=none,dashed] file {CircleDistanceGraphkappa-2.dat.tex};
\addplot [mark=none,dashed] file {CircleDistanceGraphkappa-10.dat.tex};
\addplot [mark=none,dashed] file {CircleDistanceGraphkappa-100.dat.tex};
\legend{$\kappa=1$,$\kappa=0.9$,$\kappa=0.6$,$\kappa=0$,$\kappa=-2$,$\kappa=-10$,$\kappa=-100$};

\end{axis}

\end{tikzpicture}

\end{center}
 \caption{The graph of $D_\kappa$ for some values of $\kappa
$ from $1$ at the top to $-100$ at the bottom.  The two solid lines are $\kappa=1$ and $\kappa=0$ corresponding to the intrinsic and Euclidean metrics on the circle respectively.}
 \label{fig:DistGraph}
 \end{figure}
 The graph of $D_\kappa$ for some values of $\kappa$ is plotted in Figure~
\ref{fig:DistGraph}.  Note that if $\kappa$ is taken to be sufficiently negative then the
diameter of the length $1$ circle can be made arbitrarily small.  Note as well that for
all values of $\kappa$
  \[D'_\kappa(0)=+1,\quad D'_\kappa(1)=-1,\quad\text{and}\quad
  D''_\kappa(0)=0=D''_\kappa(1).\]
These are the key facts which will be used below and they are saying that all of these
metrics are infinitesimally the same to second order.  This is related to the fact that
they all correspond to the standard \emph{Riemannian} metric on the circle.

 Just as in Section~\ref{Subsection:EucMetricOnCircle}
we can approximate $C_{\ell ,\kappa}$, a circle of length $l$ with the $\kappa$-metric
for $\kappa\in (-\infty,1]$, by using a set of $n$ points equally
spaced on the circle and equipped with the subspace metric.   Exactly
the same argument leads us to the magnitude of the length $l$ circle with the $\kappa$ metric.
\begin{thm}
For any $\kappa\in(-\infty,1]$ the magnitude of the the $\kappa$-metric length
$\ell $ circle satisfies
\[\left | C_{\ell ,\kappa}\right| :=
     \left({\int_{0}^1 e^{-\ell D_\kappa
             \left(s \right)}\text{d}s}\right)^{-1}.\]
and asymptotically it behaves as follows
  \[\left| C_{\ell ,\kappa}\right| = \ell/2 +O(\ell^{-1}) \quad\text{as }\ell\to\infty.\]
\end{thm}
\begin{proof}The calculation of the magnitude is the same as in  Section~\ref{Subsection:EucMetricOnCircle}.  Similarly, the asymptotic calculation requires
the same argument as used in the proof of Theorem~\ref{Thm:EucCircleAsymptotics}.  The relevant facts about $D_{\kappa}$ used are the following.
\begin{gather*}
D_\kappa(0)=0=D_\kappa(1); \quad
 D_\kappa(x)>0 \text{ for }x\in(0,1);\\
  D_\kappa'(0)=1=-D_\kappa'(1); \quad D_\kappa''(0)=0=D_\kappa''(1).\qedhere
 \end{gather*}
\end{proof}
Just as with Theorem~\ref{Thm:EucCircleAsymptotics} the analysis in the proof can be easily extended to show that the magnitude function
asymptotically looks like $\ell/2 +\pi^2(\kappa-1)/2\ell +\dots$.

To conclude, we note that whilst the magnitudes associated to these metrics on a circle are all different, they all have the same asymptotic behaviour.  These ideas have subsequently
been developed further in~\cite{Willerton:Homogeneous} where the magnitude of higher dimensional spheres with their intrinsic metric are calculated and the asymptotics of the magnitudes of homogenous Riemannian manifolds are given.  In that paper the calculations are done using weight measures.  Here we have confined ourselves to the more elementary, and perhaps more widely applicable, method of finite approximation.

%

\begin{thebibliography}{99}
%
\bibitem{BergerLeinster:EulerCharDivergentSeries} C.~Berger, T.~Leinster,
  \href{http://intlpress.com/HHA/v10/n1/a3/}{\textit{The Euler characteristic of a category as the sum of a divergent series}},
   Homology, Homotopy and Applications \textbf{10} (2008), 41--51.
%
\bibitem{BigginsBingham:NearConstancy}
 J.~D.~Biggins, N.~H.~Bingham,
  \textit{Near-constancy phenomena in branching processes},
  Mathematical Proceedings of the Cambridge Philosophical Society
  \textbf{110} (1991), 545-558.
%
\bibitem{Borceux:Handbook1} F.~Borceux,
  \textit{Handbook of Categorical Algebra 1: Basic Category Theory},
  Encyclopedia of Mathematics and its Applications \textbf{50}, 
  Cambridge University Press, 1994.
 %
\bibitem{Borceux:Handbook2} F.~Borceux,
  \textit{Handbook of Categorical Algebra 2: Categories and Structures},
  Encyclopedia of Mathematics and its Applications \textbf{51}, 
  Cambridge University Press, 1994.
%
\bibitem{Falconer:Book} K.~A.~Falconer,
\textit{Fractal Geometry: Mathematical Foundations and Applications},
Wiley, 2003.
%
\bibitem{HornJohnson:MatrixAnalysis}
R.~A.~Horn and C.~R.~Johnson,  
   \textit{Matrix Analysis},
    Cambridge University Press, 1985.
%
\bibitem{Hutchinson:FractalsSelfSimilarity}
J.~E.~Hutchinson,
\textit{Fractals and self-similarity},
{Indiana University Mathematics Journal}
\textbf{30} (1981), no.~5, 713--747.
%
\bibitem{KlainRota:Book} D.~A.~Klain, G.-C.~Rota,
  \textit{Introduction to Geometric Probability},
  Cambridge University Press, 1997.
  %
\bibitem{Lawvere:MetricSpaces} F.~W.~Lawvere,
  \href{http://www.tac.mta.ca/tac/reprints/articles/1/tr1abs.html}{\textit{Metric spaces, generalized logic and closed categories}};
  originally published in
  Rendiconti del Seminario Matematico e Fisico di Milano, \textbf{XLIII} (1973),  135--166;
republished in
Reprints in Theory and Applications of Categories, \textbf{1} (2002), 1--37.
%
\bibitem{Leinster:Cardinality} T.~Leinster,
  \href{http://www.math.uni-bielefeld.de/documenta/vol-13/02.html}{\textit{The Euler characteristic of a category}},
  Documenta Mathematica \textbf{13} (2008), 21--49.
%
\bibitem{Leinster:MetricSpacesBlogPost} T.~Leinster,
  \textit{Metric spaces},
  post at The $n$-Category Caf\'e, February 2008. \url{http://golem.ph.utexas.edu/category/2008/02/metric_spaces.html}
\bibitem{Leinster:MaximumEntropy} T.~Leinster,
  \textit{A maximum entropy theorem with applications to the measurement of  biodiversity},
  arXiv preprint.  \url{http://arxiv.org/abs/0910.0906v4}
%
\bibitem{Leinster:Magnitude} T.~Leinster,
  \textit{The magnitude of metric spaces},
  arXiv preprint. \url{http://arxiv.org/abs/1012.5857v3}
%
\bibitem{MacLane:Categories} S.~Mac Lane,
  \textit{Categories for the Working Mathematician},
  Second Edition,  
  Graduate Texts in Mathematics \textbf{5}, Cambridge University Press, 1998.
%
\bibitem{Meckes:PositiveDefinite} M.~Meckes,
  \textit{Positive definite metric spaces},
  arXiv preprint.
 \url{http://arxiv.org/abs/1012.5863v3}
%
\bibitem{Miller:AppliedAsymptoticAnalysis} P.~D.~Miller,
  \textit{Applied Asymptotic Analysis},
  Graduate Studies in Mathematics 75, American Mathematical Society,
2006.
%
\bibitem{SolowPolasky:MeasuringBiologicalDiversity}  A.~Solow, S.~Polasky,
    \textit{Measuring biological diversity},
    Environmental and Ecological Statistics \textbf{1} (1994), 95--107.
%
\bibitem{Speyer:ncat} D.~Speyer,
\textit{Re: Metric spaces},
reply to \cite{Leinster:MetricSpacesBlogPost}, February 2008.
\url{http://golem.ph.utexas.edu/category/2008/02/metric_spaces.html#c014950}
%
\bibitem{Willerton:Heuristics} S.~Willerton,
    \textit{Heuristic and computer calculations for the magnitude of metric spaces},
    arXiv preprint. \url{http://arxiv.org/abs/0910.5500v1}
%
\bibitem{Willerton:Homogeneous} S.~Willerton,
    \textit{On the magnitude of spheres, surfaces and other homogeneous spaces},
    arXiv preprint. \url{http://arxiv.org/abs/1005.4041v1}
%
\end{thebibliography}
\end{document}